\documentclass[oneside, 11pt]{article}
\usepackage{commands}

\begin{document}
  \title{
    Galois Groups of Trinomials with Small Coefficients
  }
  \author{Hilel Garmi}
  \date{}
  \maketitle

  \begin{abstract}
    We calculate the Galois group of $f = x^n + ax^k + b \in \mathbb{Z}\left[ x \right]$ over $\mathbb{Q}$ for all $1 \leq \left| ab \right| \leq 3$, $n$ sufficiently large and $k$ coprime to $n$.
    In particular, we prove that the Galois group of $f$ is the maximal one that its factorization type allows.
    We give a criterion for the Galois group of a factor of a trinomial to be symmetric, with an eye towards future applications.
    To the best of our knowledge, the only previously known explicit example of polynomials of degree $n$ with bounded coefficients and Galois group $S_n$ is $x^n - x - 1$, while recent results show that the probability that a random polynomial of degree $n$ with reasonably bounded coefficients has Galois group $S_n$ tends to $1$ as $n$ tends to infinity.
    Using our calculation, one can construct many new such explicit examples.
  \end{abstract}

\section{Introduction}
Let $f$ be a polynomial of the form $x^n + a x^k + b$ with $n > k > 0, a, b \in \Z \backslash \left\{ 0 \right\}$ and $1 = \gcd\left( n,k \right)$.
For each $\ell \in \N$ let $\zeta_\ell$ be a primitive $\ell$-th root of unity.
The first main result of this paper is the following:
\begin{theorem}
  \label{thm:main}
  For sufficiently large $n$, the following is true
  \[
    \begin{matrix}
      \text{If }
      \left| ab \right| \leq 3 
    & \text{ then }
      &
      \Gal\left( f \right)
      \cong
      \begin{cases}
        S_{n-2} \times S_2 
        & 
        f\left( \zeta_3 \right)f\left( \zeta_6 \right) = 0 \\
        S_{n-1} 
        & 
        f\left( 1 \right)f\left( -1 \right) = 0 \\
        S_n & \text{otherwise}
      \end{cases}
      \\ \\
      \text{If }
      \left( a,b \right) \in 
      \left\{ 
        \begin{matrix}
        \left( \pm 1, \pm 4 \right),
        \left( \pm 2, \pm 2 \right),
        \\
        \left( \pm 4, 1 \right),
        \left( \pm 1, \pm 5 \right),
        \left( \pm 5, 1 \right)
        \end{matrix}
        \right\} 
      &\text{ then }
      &
      \Gal\left( f \right) \cong S_{n}
      .
    \end{matrix}
  \]
  
  Specifically, the above holds if $n > 8$ and $f$ is irreducible, or if either 
  \begin{enumerate}
    \item $\left| ab \right| = 1$,
    \item $\left| ab \right| = 2$ and $n > 8$,
    \item $\left| ab \right| = 3$ and $n > 2^{18}$, or 
    \item $\left| ab \right| = 5$ and $n > 2^{50}$ 
  \end{enumerate}
  \end{theorem}

Informally, \Cref{thm:main} states that 
if $a$ and $b$ are small enough and $n$ is large enough, then the only obstruction to a permutation of the roots of $f$ is
for some of them to be roots of unity.

We prove \Cref{thm:main} by proving a more general result, that perhaps can be used to prove stronger versions of it in the future.
To state it, we need to introduce some more definitions and notations that are used throughout the paper.
Let $g$ be an irreducible factor of $f$ and let $L = \Q\left[ x \right] / \left( g \right)$.
Fix an algebraic closure $\overline{\Q}$ of $\Q$.
For each prime $p \in \Z$ fix an embedding of $\overline{\Q}$ into an algebraic closure $\overline{\Q}_p$ of the field of $p$-adic numbers $\Q_p$, and denote by $v_p$ the $p$-adic valuation on $\overline{\Q}_p$.
For each number field $K$ let $\rd\left( K \right)$ denote the root discriminant of $K$, i.e. the $\left[ K : \Q \right]$-th root of the absolute value of the discriminant of $K$.
The more general result can then be stated as follows: 
\begin{theorem}
  \label{thm:criterion}
  Assume that there is a constant $C_p > v_p\left( ab \right)$ for each prime $p|ab$, such that the only subextension $K \subseteq L$ satisfying both  
  \begin{enumerate}
    \item $v_p\left( \rd\left( K \right) \right) < C_p$ for each $p|ab$, and
    \item $K$ is unramified away from $ab$,
  \end{enumerate} 
  is $\Q$.
  Assume that 
  \[
    \deg g > \max_{p| ab}\left( \frac{n v_p\left( ab \right) + \exrootscountbound}{C_p} \right)
    .
  \] 
  Then the Galois group of $g$ over $\Q$ is $S_{\deg L}$.
\end{theorem} 
The reason that this result is helpful, is that there are only finitely many number fields with root discriminant below some small enough bounds, as guaranteed for example by the work of Odlyzko \cite{OdlyzkoBoundsII} and as is made even more explicit by the LMFDB \cite{lmfdb}.

A specific case of \Cref{thm:main} that is worth highlighting, is the case $f = x^n - 2x^{n-1} + 1$, in which $f$ factors into $x-1$ times the $\left( n-1 \right)$-th generalized {Fibonacci} polynomial
\[
  x^{n-1} - x^{n-2} - \ldots - x - 1,
\] 
which was proven to be irreducible by Miles \cite{Miles1960} and Wolfram \cite{Wolfram1998}.
Martin \cite{MARTIN2004213} proved that the Galois group of this polynomial is $S_{n-1}$ for all odd $n$ and conjectured the same for all even $n$.
An immediate corollary of \Cref{thm:main} is the following generalization of Martin's conjecture:
\begin{corollary}
  \label{cor:generalized_fibonacci}
  The Galois group of 
  \[
    x^{n-1} + \ldots + x^k - x^{k-1} - \ldots - 1
  \]
  is $S_{n-1}$ for all coprime $n > k > 0$ such that $n > 8$.
\end{corollary}
Martin's conjecture is that prompted our own interest in the subject through a recent related result of Bary-Soroker, Shusterman, and Zannier \cite[Appendix A]{Soroker_appendix_2017}, 

\Cref{thm:main}, \Cref{thm:criterion} and \Cref{cor:generalized_fibonacci} should all be understood in the context of recent developments in the field of explicit Hilbert irreducibility.
By Hilbert's irreducibility theorem \cite{Hilbert}, we expect a general rational polynomials of degree $n$ to have Galois group $S_n$.
Nevertheless, giving explicit examples of such polynomials for all $n$ remained a difficult problem.
Only 38 years after Hilbert's work was published, Schur \cite{schur1931affectless} showed that the Laguerre polynomials 
\[
  L_n^{\left( 0 \right)}\left( -x \right) = \sum_{k=0}^n \binom{n}{k} \frac{x^k}{k!}
\]
have Galois group $S_n$ for every $n$.
Later, the same was proven for other families of generalized Laguerre polynomials, such as the Hermite polynomials, $L^{\left( \pm 1/2 \right)}_n$ \cite{schur1931affectless}, and the reverse Bessel polynomials, $L^{\left( -2n-1 \right)}_n$ \cite{grosswald1978bessel} \cite{filaseta2002irreducibility}.

Note that in all of these examples, the height of the coefficients tend to infinity rather fast.
While we were working on this paper it was shown by Blachar, Breuillard and Kozma \cite{blachar2026galoisgroupsrandompolynomials} that the probability that the Galois group of a polynomial is $S_n$ tends to 1 as $n$ tends to infinity, even if the coefficients are sampled out of some bounded sets of integers. 
This is an improvement upon previous results in that direction, for example by Bary-Soroker, Koukoulopoulos and Kozma \cite{Bary_Soroker_2023} and by Breuillard and Varj\'{u} \cite{Breuillard_2019}.
These previous results provided some of the motivation for this work, as it appears that the only known example with a bounded height is the family of Selmer trinomials $f_n\left( x \right) = x^n - x - 1$ , which was shown to be irreducible by Selmer \cite{SelmerIrreducibility1956} and to have Galois group $S_n$ by Nart and Vila \cite{NartVila1979}. 
The mentioned family of $n$-th generalized Fibonacci polynomials is just one out of many families that we now know to have symmetric Galois group thanks to \Cref{thm:main}. 

\subsection{Proof Overview}
Our work can be seen as a complement of the work of Osada \cite{Osada1987TheGG}, which provided conditions that guarantee that $\Gal\left( f \right) \simeq S_{\deg f}$.
Osada's idea was that since a transitive permutation group that is generated by transpositions is symmetric, and since, as a consequence of Hermite-Minkowski Theorem, the Galois group of $f$ is generated by its inertia groups, it is enough to show that the inertia group at any non-Archimedean place acts on the roots of $f$ either trivially or by a single transposition.
In his work, Osada imposed some conditions on $f$ under which he was able to do exactly that.

The work of Martin \cite{MARTIN2004213} used the same approach to prove his aforementioned result for odd $n$, and its failure for even $n$ provides an instructive example of the limitations of this approach.
Indeed, for even $n$, 
\[
  x^n - 2x^{n-1} + 1 \equiv x^n + 1 \pmod{2}
\] 
and hence the inertia group of a place above $2$ acts on the roots of $f$, or of some $g$ dividing $f$ for that matter, non-trivially and by a single transposition. 
However, at any non-Archimedean place away from $ab$, we can, and in this paper do, use Osada's approach to show that the inertia group acts on the roots of $f$ either trivially or by a single transposition.
In this situation, rather than implying that the Galois group of $g$ over $\Q$ is symmetric, Hermite-Minkowski Theorem only implies that the Galois group of $g$ over the maximal subextension of $L = \Q\left[ x \right] / \left( g \right)$ that is unramified away from $ab$ is symmetric .
Thus, for the Galois group of $g$ over $\Q$ to be symmetric, it is enough to show that the only such subextension is $\Q$ itself.

As is explained in more detail in \Cref{sec:global}, the above analysis implies that we can prove \Cref{thm:criterion} by proving a small enough bound on the $p$-adic valuation of the discriminant of $L$ for any prime $p | ab$.
This is exactly what we do in \Cref{sec:local}, by analyzing the valuation of the different ideal of a field extension generated by a root of $f$ over a general non-Archimedean local field.
For conciseness, we chose to present the necessary result for the primes that do not divide $ab$, as a consequence of this much more complicated analysis that was originally written to tackle the problematic primes that divide $ab$.
However, we note that the same result could have been proven in a much more direct manner, similar to the way Osada proved \cite[Lemma 3]{Osada1987TheGG}.

To derive \Cref{thm:main} from \Cref{thm:criterion}, we prove in \Cref{sec:examples} different obstructions for the embeddability of number fields that are unramified away from $ab$ and have a small root discriminant, as well as a few lemmas that help us analyze the factorization of $f$ over $\Q$.

\subsection{Acknowledgements and tools}
I would like to thank my advisor, Borys Kadets, for many helpful remarks and suggestions along the way.
I used mainly SageMath and occasionally Magma for calculations and verifications.
No LLM was used to derive any of the results presented in this paper.
I did use GitHub Copilot for autocompletion and Gemini Pro for code generation, to find relevant literature and to translate some of it into English.
The author was supported by the Israel Science Foundation (ISF), Grant No. 2112/26 (to Borys Kadets)

  \section{
    \label{sec:local}
    Local Analysis
  }
  \Cref{cor:inertia_group_acts_as_transposition_external_settings} and \Cref{prop:discriminant_bound} are the only results that we use from \Cref{sec:local} in later sections, and thus a reader who would like to skip the technical details of the local analysis can do so and still be able to follow the rest of the paper.

  Let $\F$ be a non-Archimedean local field with local ring $R$, and residue field $\kappa$. 
  Let $\overline{\F}$ be an algebraic closure of $\F$ and let $\overline{\kappa}$ be the residue field of $\overline{\F}$.
  Let $v$ be an extension to $\overline{\F}$ of the unique valuation of $\F$ for which $v\left( K \right) = \Z$.
  Let 
  \[
    \h{} = a_n x^n + a_m x^m + a_k x^k \in \F\left[ x, x^{-1} \right]
  \]
  be a trinomial with $a_n, a_m, a_k \neq 0$ and 
  \[
    1 = \gcd\left( n - m, m - k, k - n \right) = \gcd\left( n - m, m - k \right).
  \]
  We use $\left( \h{} \right)$ to denote the principal ideal generated by $\h{}$ in the ring $\F\left[ x, x^{-1} \right]$.
  Note that no assumption is made about the order of $n,m$ and $k$ and thus all statements about them are true for all of the permutations of $n,m$ and $k$.
  In light of this, in the sequel we often state or define something for some permutation of $n,m$ and $k$ and then use it for other permutations as well. 
  
  Let $\alpha \in \overline{\F}^\times$.
  Note that if $\alpha$ is a root of $\h{}$, then
  \[
    \alpha \h{}'\left( \alpha \right) 
    = 
    \alpha \h{}'\left( \alpha \right) - n\h{}\left( \alpha \right)
    =
    (m - n) a_m \alpha^m + (k - n) a_k \alpha^k.
  \] 
  Thus, let us denote
  \[
    \h{n} 
    =
    x\h{}' - n\h{} 
    = 
    (m - n) a_m x^m + (k - n) a_k x^k
  \]
  For most cases, we can bound $v\left( \h{}'\left( \alpha \right) \right)$ using the non-Archimedean triangle inequality and the summation term of least valuation of either $\h{n}$, $\h{m}$ or $\h{k}$.
  We will try to first characterize the cases where this strategy fails.
  
  \subsection{Magnetic roots}
  So that we can speak of each of the summation terms of $\h{n}$ separately, let us denote
  \[
    \h{n,m} = (m - n) a_m x^m.
  \]
  Note that 
  \begin{align}
    \label{eq:valuation_inequality_condition}
    & v\left( \alpha\h{n}\left( \alpha \right) \right) 
    > 
    \min\left\{  
      v\left( \h{n,m}\left( \alpha \right) \right), 
      v\left( \h{n,k}\left( \alpha \right)\right)
    \right\}
  \end{align}
  if and only if
  \begin{equation}
    \label{eq:valuation_positivity_condition}
    v\left( 1 + \frac{\h{n,m}\left( \alpha \right)}{\h{n,k}\left( \alpha \right)} \right) > 0,
  \end{equation}
  
  i.e. if and only if 
  \begin{equation}
    \label{eq:principal_unit_condition}
    -\frac{\h{n,m}\left( \alpha \right)}{\h{n,k}\left( \alpha \right)}
    =
    -\frac{
      \left( m - n \right)a_m \alpha^m
    }{
      \left( k - n \right)a_k \alpha^k
    }
    \text{ is a principal unit.}
  \end{equation}
  To state the last condition in a more useful language, we denote 
  \[
    \q{n}{m} 
    :=
    -\frac{
      \left( m - n \right)a_m
    }{
      \left( k - n \right)a_k
    } 
  \]
  and make the following
  \begin{definition}
    \label{def:pi}
    Let $\pu$ be the groups of principal units in $\overline{\F}$, i.e the groups of elements in $\overline{\F}$ whose difference from $1$ has positive valuation. 
    Let 
    \[
      \pi: \overline{\F}^\times \to \frac{\overline{\F}^\times}{\pu}
    \]
    be the canonical quotient map. 
  \end{definition} 
  In this language condition \cref{eq:principal_unit_condition} can be stated as
  \begin{equation}
    \label{eq:mod_principal_units_condition}
    \pi\left(  
      \q{n}{m}
    \right) = 
    \pi\left( \alpha^{k - m} \right)
  \end{equation}
  Since $\F\left[ x, x^{-1} \right]^\times = \F^\times x^{\Z}$, if \cref{eq:mod_principal_units_condition} holds for $h{}$ then it also holds for any other generator of $\left( \h{} \right)$.
  Thus, we can make the following
  \begin{definition}
    If the equivalent conditions \cref{eq:valuation_inequality_condition}, \cref{eq:valuation_positivity_condition}, \cref{eq:principal_unit_condition} and \cref{eq:mod_principal_units_condition} hold, we say that $\pi\left( \alpha \right)$ is an \textbf{\es{n}} disk of $\left( \h{} \right)$.
    We say that $\pi\left( \alpha \right)$ is \textbf{\anex} disk of $\left( \h{} \right)$
    if it is \es{n}, \es{m} and \es{k}.
    In many cases we write $\h{}$ instead of $\left( \h{} \right)$, and if $\alpha$ is also a root of $\h{}$ we say that it is \textbf{\anex (respectively an \es{n}) root of $\h{}$}.
  \end{definition}
  We use the analogy of magnetism because, as explained above, for a magnetic root $\alpha$, the valuation of $\h{}'\left( \alpha \right)$ is bigger than, i.e., in the metric induced by $v$, $\alpha$ is on average closer to its conjugates than, what one might expect by bounding the valuations of $\h{i,j}$ for each pair of distinct $i,j \in \left\{ m,n,k \right\}$.
  We use the term disk because each fiber of $\pi$ is a disk in the metric induced by $v$. 

  \begin{lemma}
    \label{lem:scalability}
    Assume that $\pi\left( \alpha \right)$ is \anex disk of $\h{}$.
    Fix $a \in \F^\times$ and $\delta =\pm 1$.
    Then $\pi\left( a\alpha^\delta \right)$ is \anex disk of $\h{}\left( x^\delta / \alpha \right)$.
  \end{lemma}
  \begin{proof}
    Let $u\left( x \right)$ be $a x^\delta $.
    Then $\h{} \circ u$ is a trinomial with coefficients $\delta m, \delta n, \delta k$ and.
    One can check that 
    \[
      \left( \h{} \circ u^{-1} \right)_{\delta m,\delta n}\left( u\left( \alpha \right) \right) = \h{m,n}\left( \alpha \right)
    \]
    and the rest follows from opening the definition of \ex disks.
  \end{proof}

  \begin{lemma}
    \label{lem:two_escaping_implies_ex}
    If $\pi\left( \alpha \right)$ is \es{n} and \es{m} disk of $\h{}$, 
    then it is \anex disk of $\h{}$.
  \end{lemma}
  \begin{proof}
    It can be verified that 
    \[
      1 = \q{n}{m} \q{m}{k} \q{k}{n}
    \]
    
    and hence, the result follows from \cref{eq:mod_principal_units_condition}.
    
    \end{proof}
  \begin{corollary}
    \label{cor:ex_disk}
    There are $r,s \in \Z$ depending only on $n-m$ and $m-k$, such that $\mu \in \frac{\overline{\F}^\times}{\pu}$ is \anex disk of $\h{}$ if and only if it solves the $\Z$-linear equation
    \begin{equation}
      \label{eq:ex}
      \begin{bmatrix}
        1 \\
        0
      \end{bmatrix} x 
      = 
      \begin{bmatrix}
        \pi\left( \q{n}{m}^r \q{m}{k}^s \right) \\
        \pi\left( \q{n}{m}^{-\left( n - k \right)} \q{m}{k}^{k - m} \right)
      \end{bmatrix} .
    \end{equation}
  \end{corollary}
  \begin{proof}
    Another way to state \Cref{lem:two_escaping_implies_ex} is that $\mu$ is \anex disk of $\h{}$ if and only if it solves the $\Z$-linear equation
    \begin{equation}
      \label{eq:raw_ex_equation}
      \begin{bmatrix}
        k - m \\
        n - k
      \end{bmatrix} x 
      = 
      \begin{bmatrix}
        \pi\left( \q{n}{m} \right) \\
        \pi\left( \q{m}{k} \right)
      \end{bmatrix}.
    \end{equation}
    Since by assumption $k - m$ and $n - k$ are coprime, by the Euclidean algorithm there are $r,s \in \Z$ such that 
    \[
      r(k - m) + s(n - k) = 1.
    \]
    Therefore, multiplying both sides of \cref{eq:raw_ex_equation} by 
    \[
      \begin{bmatrix}
        r & s \\
        -\left( n - k \right) & k - m
      \end{bmatrix}
    \]
    shows that it is equivalent to the required equation.
  \end{proof}
  
  \begin{lemma}
    \label{lem:reduction}
    If $\pi\left( \alpha \right)$ is \anex disk of $\h{}$, then there is a $u \in \F^\times x^{\pm 1}$ such that : 
    \begin{enumerate}
      \item 
      $\pi\left( 1 \right)$ is \anex disk of $\h{} \circ u^{-1}$.
      \item 
      $\h{} \circ u^{-1}$ is integral over $R$.
    \end{enumerate}
  \end{lemma}
  \begin{proof}
    According to \Cref{cor:ex_disk}, $\pi\left( \alpha \right) = \q{n}{m}^r \q{m}{k}^s$ when $r$ and $s$ are as defined in \Cref{cor:ex_disk}.
    Set $q = \q{n}{m}^r \q{m}{k}^s$.
    Therefore, $\pi\left( q^{-1}\alpha \right) = 1$ and in particular $v\left( q^{-1}\alpha \right) = 0$.
    Since $\h{}\left( qx \right)$ is a trinomial, its Newton polygon has at most two segments, and hence the set of valuations of its roots is of cardinality at most $2$.
    Hence, as we have already shown that $0$ is in the set of valuations of roots of $\h{}$, it follows that this set is either non negative or non positive.
    We can set $\delta = 1$ in the first place and $\delta = -1$ in the second place (and either of them if both cases hold) and set $u = q x^\delta$.
    Then $\h{} \circ u^{-1}$ is integral over $R$ and $\pi\left( 1 \right)$ is \anex disk of $\h{} \circ u^{-1}$ according to \Cref{lem:scalability}.
  \end{proof}
  
  \begin{lemma}
    \label{lem:ex_lift}
    If $\h{}$ has \anex disk then it has \anex root.  
  \end{lemma}
  \begin{proof}
    By \Cref{lem:reduction} we can assume without loss of generality that $\h{}$ is integral over $R$ and that $\pi\left( 1 \right)$ is \anex disk of it.
    In this situation, it suffices to show that $1$ lifts to a root of $\h{}$, or equivalently, by the integrality of $\h{}$, that $\h{}\left( 1 \right) = 0 \in \kappa$.
    Indeed, by \cref{eq:mod_principal_units_condition}
    \begin{align*}
      \h{}\left( 1 \right) 
      = 
      & \
      a_n + a_m + a_k
      =
      a_n + a_m \q{k}{n} + a_k \q{m}{n}
      =
      a_n - \frac{\left( n - k \right)a_n}{\left( m - k \right)} - \frac{\left( n - m \right)a_n}{\left( k - m \right)} 
      = 
      0
      .\qedhere
    \end{align*}\qedhere
  \end{proof}

  \subsection{Reduced magnetic roots}
  In light of \Cref{lem:reduction} we make the following
  \begin{definition}
    \label{def:aredex}
    We say that $\alpha$ ia \textbf{\aredex root} of $\h{}$ and that $\h{}$ is \\ \textbf{\aredex polynomial} if:
    \begin{enumerate}
      \item $\alpha$ is \anex root of $h$.
      \item $\pi\left( \alpha \right) = 1$. 
      \item $\h{}$ is integral and primitive over $R$.
    \end{enumerate}
  \end{definition} 
  
  Note that if $\alpha$ is \aredex root of $\h{}$, then the Newton polygon of $\h{}$ has a segment of slope $0$.
  Therefore, no loss of generality arises from replacing the condition that $\alpha$ is \anex root of $\h{}$ with the condition of the following
  \begin{lemma}
    \label{lem:slope_length_valuation}
    Assume that $\alpha$ is \aredex root of $\h{}$ and that the segment connecting the points $\left( k, v\left( a_k \right) \right)$ and $\left( m, v\left( a_m \right) \right)$ belong to the Newton polygon of $\h{}$ and has slope $0$.
    Then
    \begin{enumerate}
      \item \label{item:smaller_slope_length_valuation} 
      $v\left( n - k \right) = v\left( n - m \right) = 0$.
      \item \label{item:bigger_slope_length_valuation} 
      $v\left( m - k \right) 
        = 
        v\left( a_n \right)
        .$
    \end{enumerate}
  \end{lemma}
  \begin{proof}
    \cref{eq:valuation_inequality_condition} in the definition of $k$-magnetism and $ m $-magnetism implies the equalities 
    \begin{equation}
      \label{eq:raw_valuation_equality}
      v\left( \h{k,n}\left( \alpha \right) \right) = 
      v\left( \h{k,m}\left( \alpha \right) \right)
      \quad \text{and} \quad
      v\left( \h{m,n}\left( \alpha \right) \right) = 
      v\left( \h{m,k}\left( \alpha \right) \right)
      .
    \end{equation}
    To simplify it, note that by our assumptions regarding the Newton polygon of $\h{}$ it follows that $v\left( a_n \right) > v\left( a_m \right) = v\left( a_k \right)$, and that since we also assume that $\h{}$ is \aredex polynomial and thus integral and primitive over $R$, it follows that $v\left( a_m \right) = v\left( a_k \right) = 0$. 
    Substituting these two facts into \cref{eq:raw_valuation_equality}, as well as the assumption that $v\left( \alpha \right) = 0$, gives the equalities
    \begin{gather}
      v\left( m - k \right) = v\left( n - k \right) + v\left( a_n \right) 
      \label{eq:first_relation_between_slope_length_valuations}
      \\
      \notag
      \text{ and } 
      \\
      v\left( m - k \right) = v\left( n - m \right) + v\left( a_n \right) 
      \label{eq:second_relation_between_slope_length_valuations}
    \end{gather}
    which together imply that $v\left( n - k \right) = v\left( n - m \right)$.
    Since $1 = \gcd\left( n - m, n - k \right)$, it hence follows that $v\left( n - k \right) = v\left( n - m \right) = 0$, proving \Cref{item:smaller_slope_length_valuation}.
    Substituting back in either \cref{eq:first_relation_between_slope_length_valuations} or \cref{eq:second_relation_between_slope_length_valuations} proves \Cref{item:bigger_slope_length_valuation}.
  \end{proof}
  According to \Cref{item:bigger_slope_length_valuation} of \Cref{lem:slope_length_valuation} we can classify the \redex polynomials into two kinds:

  \subsubsection{The equal valuation case}
  Throughout this subsection assume that $v\left( a_n \right) = v\left( a_m \right) = v\left( a_k \right) = 0$ and that $\h{}$ is a reduced magnetic polynomial.

  \begin{lemma}
    \label{lem:magnetism_away_from_ab}
    $\alpha$ is \anex root of $\h{}$ if and only if $v\left( \h{}'\left( \alpha \right) \right) > 0$.
  \end{lemma}
  \begin{proof}
    Since the Newton polygon of $\h{}$ has only segments of slope $0$, $v\left( \alpha \right) = 0$.
    Therefore, since $1 = \gcd\left( m - n, k - n \right)$, it follows that 
    \[
      \min\left\{  
        v\left( \h{n,m}\left( \alpha \right) \right),
        v\left( \h{n,k}\left( \alpha \right) \right)
      \right\} = 0
    \]
    and hence the result follows from \cref{eq:valuation_inequality_condition} defining magnetism.
  \end{proof}

  \begin{lemma}
    \label{lem:2_magnetic_roots}
    $\h{}$ has at most $2$ magnetic roots.
  \end{lemma}
  \begin{proof}
    By \Cref{cor:ex_disk} all of its \ex roots lie above $1$.
    Therefore, it suffices to show that $1 \in \kappa$ is not a triple root of $\h{}$.
    However if it was it would have been a double root of 
    \[
      \h{n} 
      = 
      x\h{}' - n\h{} 
      =
      x^k \left( a_m x^{m - k} + a_k \right) 
    \] 
    contradicting the reducedness of $\left( a_m x^{m - k} + a_k \right) \leq \kappa\left[ x, x^{-1} \right]$ which follows form the fact that $v\left( m - k \right) = 0$ which follows in turn from \Cref{lem:slope_length_valuation} whose conditions we assume above.
  \end{proof}

  \begin{corollary}
    \label{cor:inertia_group_acts_as_transposition}
    The inertia group acts on the roots of $\h{}$ either trivially or as a single transposition.
  \end{corollary}
  \begin{proof}
    It suffices to show that $\h{}$ has at most $2$ roots $\alpha$ at which $v\left( \h{}'\left( \alpha \right) \right) > 0$, or equivalently, according to \Cref{lem:magnetism_away_from_ab}, to show that $\h{}$ has at most two \ex roots, which is true by \Cref{lem:2_magnetic_roots}. 
  \end{proof}
  \begin{corollary}
    \label{cor:inertia_group_acts_as_transposition_external_settings}
    Let $p$ be a rational prime.
    If $v_p\left( ab \right) = 0$ then the inertia group at any place above $p$ acts on the roots of $f$ either trivially or as a single transposition.
  \end{corollary}
  \subsubsection{The non equal valuation case}
  Assume now that $\h{}$ is a \redex polynomial with    
  \[
    v\left( m - k \right) = v\left( a_n \right) > 0.
  \]
  It follows in particular that $p := {\rm char}\left( \kappa \right) > 0$ and we assume moreover that $\kappa$ is perfect.
  \begin{lemma}
    \label{lem:counting_ex_roots}
    $\h{}$ has exactly $p^{v\left( m - k \right)} = p^{v\left( a_n \right)}$ \ex roots.
  \end{lemma}
  \begin{proof}
    When considered as a polynomial over $\kappa$, $\h{}$ satisfies 
    \[
      \h{} 
      = 
      x^k \left( 
        a_m x^{m - k} + a_k 
      \right) 
      =
      x^k \left(  
        a_m^{\left( p^{-v\left( m - k \right)} \right)} x^{\frac{m - k}{p^{v\left( m - k \right)}}} + a_k^{\left( p^{-v\left( m - k \right)} \right)}
      \right)^{p^{v\left( m - k \right)}}
      \in \kappa\left[ x \right]
    \]
    and $\left( a_m^{\left( p^{-v\left( m - k \right)} \right)} x^{\frac{m - k}{p^{v\left( m - k \right)}}} + a_k^{\left( p^{-v\left( m - k \right)} \right)} \right) \subseteq \kappa\left[ x, x^{-1} \right]$ is reduced.
    Therefore, since we assume here that $\h{}$ is \aredex polynomial and in particular integral over $R$, it follows that $\h{}$ has exactly $p^{v\left( m - k \right)}$ roots in $\overline{\F}$ lying over $1$, which means, according to \Cref{cor:ex_disk}, that $\h{}$ has exactly $p^{v\left( m - k \right)}$ \ex roots.
  \end{proof}

  Let $p$ be a rational prime.
  We denote by $\D_\alpha$ the different ideal of $\Q_p\left( \alpha \right)$ for each $\alpha \in \overline{\Q_p}$ and deduce from \Cref{lem:counting_ex_roots} the following

  \begin{corollary}\label{cor:ex_bound}
    Assume that $p| ab$.
    Then $f$ has at most $\exrootscountbound$ \ex roots above $\Q_p$ and for each such root $\alpha$, 
    \[
      v_p\left( \Deta \right) \leq \diffbnd.
    \]
  \end{corollary}
  \begin{proof}
    Let $\alpha$ be \anex root of $f$ (if there is no such root then the result is trivial).
    According to \Cref{cor:ex_disk}, $\pi\left( \alpha \right)$ is the only \ex disk and it belongs to $\pi\left( \Q^\times \right)$ (recall \cref{def:pi} for $\pi$).
    In particular, all of the \ex roots have the same valuation and that valuation is an integer.
    If that integer is at least $1$, then the number of \ex roots can not exceed the maximal height difference between two vertices of the Newton polygon of $f$, which is clearly less than $\exrootscountbound$.
    Assume then that the valuation of the \ex roots is $0$.
    Then, by \Cref{lem:scalability} we can assume without loss of generality that $f$ is \aredex polynomial.
    In particular exactly one of the coefficients $a$ and $b$ has positive valuation, and according to \Cref{lem:counting_ex_roots}, the number of \ex roots of $f$ is bounded by $p$ raised to the valuation of that coefficient, which is exactly $\exrootscountbound$.

    In particular it follows that $\alpha$ has at most $\exrootscountbound$ conjugates above $\Q_p$ and hence its ramification degree is at most $\exrootscountbound$.
    Denote this ramification degree by $e$.
    Then, $1-e^{-1} + v_p\left( e \right)$ is a well known upper bound on the valuation of $\Deta$, the different of $\Q_p\left( \alpha \right)$ over $\Q_p$ (see for example \cite[Remark 1 to Proposition 13 of Chapter. III \S 7]{SerreLocalFields}).
    However, $$1 - e^{-1} + v_p\left( e \right) \leq 1 + v_p\left( e \right) \leq 1 + v_p\left( \exrootscountbound \right) = 1 + \exrtscntbndval$$ and hence the result follows.
  \end{proof}

  \subsection{The derivative at non \ex roots}
  \begin{lemma}
    \label{lem:non_es_bound}
    If $\alpha$ is a non \es{n} root of $\h{}$ then
    \[
      v\left( \alpha \h{}'\left( \alpha \right) \right) 
      \leq
      \max\left\{ a_m \alpha^m, a_k \alpha^k \right\} 
    \]
  \end{lemma}
  \begin{proof}
    From the definition of \es{n} roots it follows that 
    \[
      v\left( \alpha \h{}'\left( \alpha \right) \right) 
      \leq
      \min\left\{ \left( m - n \right)a_m \alpha^m, \left( k - n \right)a_k \alpha^k \right\}
    \]
    and hence the result follows from the assumption that 
    $
      1 = \gcd\left( m - n, k - n \right)
    $
    and hence either $v\left( m-n \right) = 0$ in which case $v\left( \alpha \h{}'\left( \alpha \right) \right) \leq v\left( a_m \alpha^m \right)$ or $v\left( k - n \right) = 0$ in which case $v\left( \alpha \h{}'\left( \alpha \right) \right) \leq v\left( a_k \alpha^k \right)$.
  \end{proof}

  \begin{corollary}
    \label{cor:non_ex_bound}
    If $\alpha$ is a non \ex root of $\h{}$ then
    \[
      v\left( \alpha \h{}'\left( \alpha \right) \right) 
      \leq
      \max\left\{  
        v\left( a_n \alpha^n \right),
        v\left( a_m \alpha^m \right),
        v\left( a_k \alpha^k \right)
      \right\}.
    \]
  \end{corollary}
  \begin{proof}
    By \Cref{lem:two_escaping_implies_ex} we can assume without loss of generality that $\alpha$ is neither an \es{n} root of $\h{}$ nor an \es{m} root of $\h{}$.
    Therefore, it follows from \Cref{lem:non_es_bound} that 
    \begin{align*}
      v\left( \alpha \h{}'\left( \alpha \right) \right)
      \leq
      &
      \min\left\{  
        \max\left\{  
          a_m \alpha^m,
          a_k \alpha^k
        \right\},  
        \max\left\{  
          a_n \alpha^n,
          a_k \alpha^k
        \right\}
      \right\}
      \leq \\
      & \max\left\{  
        \max\left\{  
          a_m \alpha^m,
          a_k \alpha^k
        \right\},  
        \max\left\{  
          a_n \alpha^n,
          a_k \alpha^k
        \right\}
      \right\} \leq \\
      & \max\left\{  
        a_n \alpha^n,
        a_m \alpha^m,
        a_k \alpha^k
      \right\}
      .
    \end{align*} 
  \end{proof}

  Just like in the case of \Cref{lem:slope_length_valuation}, the condition that $\alpha$ is a non \ex root of $\h{}$ is just as general as the condition of the following
  \begin{corollary}
    \label{cor:explicit_non_ex_bound}
    If $\alpha$ is a non \ex root of $\h{}$, and the segment connecting the points $\left( m, v\left( a_m \right) \right)$ and $\left( n, v\left( a_n \right) \right)$ belong to the Newton polygon of $\h{}$ and its slope is $-v\left( \alpha \right)$, then
    \begin{gather}
      v\left( \alpha \h{}'\left( \alpha \right) \right) 
      \leq
      v\left( a_k \right) + \frac{k}{n - m}v\left( \frac{a_m}{a_n} \right)
      \label{item:k_es}
    \end{gather}
    and 
    \begin{gather}
      \text{either }
      \quad
      v\left( n - m \right) 
      >
      0
      \quad
      \text{ or }
      \quad
      v\left( \alpha \h{}'\left( \alpha \right) \right)
      \leq 
      \frac{
        n v\left( a_m \right) - m v\left( a_n \right)
      }{
        n - m
      }.
      \label{item:non_k_es}
    \end{gather}
  \end{corollary}
  \begin{proof}
    For each $j \in \left\{ k,n,m \right\}$, $v\left( a_j \alpha^j \right)$ is the $y$-coordinate of the intersection with the $y$-axis of the line of slope $-v\left( \alpha \right)$ passing through the point $\left( j, v\left( a_j \right) \right)$.
    Since the Newton polygon of $\h{}$ contains the segment connecting the points $\left( m, v\left( a_m \right) \right)$ and $\left( n, v\left( a_n \right) \right)$, the point $\left( k, v\left( a_k \right) \right)$ lies above the line containing this segment.
    However, the slope of this line is $-v\left( \alpha \right)$, and hence among the lines of slope $-v\left( \alpha \right)$ passing through the points $\left( j, v\left( a_j \right) \right)$ for $j \in \left\{ k,n,m \right\}$, the one passing through $\left( k, v\left( a_k \right) \right)$ has the largest intersection with the $y$-axis.
    Therefore, according to \Cref{cor:non_ex_bound}, $v\left( \alpha \h{}'\left( \alpha \right) \right) \leq v\left( a_k \alpha^k \right)$, and it can be verified that $v\left( a_k \alpha^k \right)$ equals to the bound presented in \cref{item:k_es}.

    To prove \cref{item:non_k_es}, assume that $v\left( n - m \right) = 0$ and we will prove that the required bound on $v\left( \alpha \h{}'\left( \alpha \right) \right)$ holds.
    Since $\alpha$ is assumed to be non-magnetic, according to \Cref{lem:two_escaping_implies_ex} it is either \es{m} or \es{n}.
    By the criterion \cref{eq:valuation_inequality_condition} for \es{n}, it follows that in the first case 
    $$
      v\left( \alpha \h{}'\left( \alpha \right) \right) 
      \leq 
      v\left( \h{m,n}\left( \alpha \right) \right)
      =
      v\left( \left( n - m \right) a_n \alpha^n \right)
      =
      v\left( a_n \alpha^n \right)
    $$
    and in the second case
    $$
      v\left( \alpha \h{}'\left( \alpha \right) \right) 
      \leq 
      v\left( \h{n,m}\left( \alpha \right) \right)
      =
      v\left( \left( m - n \right) a_m \alpha^m \right)
      =
      v\left( a_m \alpha^m \right)
    $$
    and it can be verified that both $v\left( a_n \alpha^n \right)$ and $v\left( a_m \alpha^m \right)$ are equal to the bound on $v\left( \alpha\h{}'\left( \alpha \right) \right)$ presented in \cref{item:non_k_es}.
  \end{proof}
  \subsection{The discriminant}
  Fix a rational prime $p$.
  Recall that we denote by $g$ an arbitrary factor of $f$, by $L$ the field $\Q\left[ x \right] / \left( g \right)$ and by $\Deta$ the different ideal of the image of $L$ in $\overline{\Q_p}$ under the embedding that maps $x$ to $\alpha$.
  \begin{proposition}
    \label{prop:discriminant_bound}
    Assume that $p | ab$.
    Then
    \[
      v_p\left( \disc\left( L \right) \right)
      \leq
      nv_p\left( ab \right) + \exrootscountbound
    \]
    and if the Newton polygon of $f$ consists of two segments of different slopes then  
    \[
      v_p\left( \disc\left( L \right) \right)
      \leq
      nv_p\left( b \right) + \exrootscountbound
    \]
    \end{proposition}
  \begin{proof}
    We divide now into three cases.
    Assume first that the Newton polygon of $f$ consists of a single segment.
    Then, by \cref{item:k_es} of \Cref{cor:explicit_non_ex_bound}, for each non \ex root $\alpha$ of $f$,
    \[
      v_p\left( \Deta \right) 
      \leq
      v_p\left( \alpha f'\left( \alpha \right) \right)
      \leq
      v_p\left( a \right) + \frac{k}{n} v_p\left( b \right)
      \leq
      v_p\left( ab \right)
    \]
    and therefore by \Cref{cor:ex_bound},    
    \begin{align*}
      &
      v_p\left( \disc\left( L \right) \right)
      =
      \sum_{\alpha} v_p\left( \Deta \right)
      \leq 
      \\
      &
      \left( n - \exrootscountbound \right) v_p\left( ab \right)
      +
      \exrootscountbound \left( 1 + v_p\left( ab \right) \right)
      \leq
      \\
      &
      nv_p\left( ab \right) + \exrootscountbound
      .
    \end{align*}

    Assume now that the Newton polygon of $f$ consists of two segments of different slopes and that $v_p\left( k \right) = 0$.
    Then, according to \cref{item:k_es} of \Cref{cor:explicit_non_ex_bound}, for each non \ex root $\alpha$ of $f$ whose valuation is equal to minus the slope of the rightmost segment of the Newton polygon of $f$,
    \[
      v_p\left( \Deta \right) 
      \leq
      v_p\left( \alpha f'\left( \alpha \right) \right)
      \leq
      v_p\left( b \right) + \frac{0}{n - k}v_p\left( a \right)
      =
      v_p\left( b \right)
    \]
    and according to \cref{item:non_k_es} of \Cref{cor:explicit_non_ex_bound} for each non \ex root $\alpha$ of $f$ whose valuation is equal to minus the slope of the leftmost segment of the Newton polygon of $f$,
    \[
      v_p\left( \Deta \right) 
      \leq
      v_p\left( \alpha f'\left( \alpha \right) \right)
      \leq
      \frac{k v_p\left( b \right) - 0 v_p\left( a \right)}{k - 0}
      =
      v_p\left( b \right)
    \]
    Therefore, according to \Cref{cor:ex_bound},    
    \begin{align*}
      &
      v_p\left( \disc\left( L \right) \right)
      = 
      \sum_{\alpha} v_p\left( \Deta \right)
      \leq
      \\
      &
      \left( n - \exrootscountbound \right) v_p\left( b \right)
      +
      \exrootscountbound \left( 1 + v_p\left( b \right) \right)
      \leq
      \\
      &
      nv_p\left( b \right) + \exrootscountbound
      .
    \end{align*}

    Assume finally that the Newton polygon of $f$ consists of two segments of different slopes and that $v_p\left( n - k \right) = 0$.
    Note that in this case according to \Cref{lem:slope_length_valuation} the valuation of the \ex roots of $f$, if they exist, is minus the slope of the leftmost segment of the Newton polygon of $f$. 
    Therefore, according to \cref{item:non_k_es} of \Cref{cor:explicit_non_ex_bound}, for each root $\alpha$ of $f$ whose valuation is equal to minus the slope of the rightmost segment of the Newton polygon of $f$,
    \[
      v_p\left( \Deta \right) 
      \leq
      v_p\left( \alpha f'\left( \alpha \right) \right)
      \leq
      \frac{
        n v_p\left( a \right) - kv_p\left( 1 \right)
      }{
        n - k
      }
      =
      \frac{n}{n - k}v_p\left( a \right)
    \]
    Similarly, according to \cref{item:k_es} of \Cref{cor:explicit_non_ex_bound} for each non \ex root $\alpha$ of $f$ whose valuation is equal to minus the slope of the leftmost segment of the Newton polygon of $f$,
    \[
      v_p\left( \Deta \right) 
      \leq
      v_p\left( \alpha f'\left( \alpha \right) \right)
      \leq
      v_p\left( 1 \right) + \frac{n}{k} v_p\left( \frac{b}{a } \right)
      =
      \frac{n}{k} v_p\left( \frac{b}{a } \right)
    \]
    Therefore, according to \Cref{cor:ex_bound},    
    \begin{align*}
      &
      v_p\left( \disc\left( L \right) \right)
      =
      \sum_{\alpha} v_p\left( \Deta \right)
      \leq
      \\
      &
      \left( n - k \right)\frac{n}{ n - k } v_p\left( a \right)
      +
      \left( k - \exrootscountbound \right) \frac{n}{k} v_p\left( b / a \right)
      +
      \\
      &
      \exrootscountbound \left( 1 + v_p\left( b / a \right) \right)
      \leq
      \\
      &
      nv_p\left( a \right) + k v_p\left( b / a \right) + \exrootscountbound
      \leq
      \\
      &
      nv_p\left( b \right) + \exrootscountbound
      \ .
      \qedhere
    \end{align*}
    \end{proof}

  \section{
    \label{sec:global}
    Global Implications
  }
  Recall that we use $g$ to denote an arbitrary factor of $f$ and $L$ to denote the field $\Q\left[ x \right] / \left( g \right)$.
  We use $N$ to denote the product of the inertia groups of all of the valuations of $\overline{\Q}$ vanishing at $ab$.
  \Cref{cor:inertia_group_acts_as_transposition_external_settings} allows us to prove the following corollary:
  \begin{corollary}
    \label{cor:transitive_action}
    If $N$ acts transitively on $\Hom_{\Q}\left( L, \overline{\Q} \right)$, then the Galois group of $g$ over $\Q$ is $S_{\deg g}$.
  \end{corollary}
  \begin{proof}
    It follows from \Cref{cor:inertia_group_acts_as_transposition_external_settings} and the fact that any transitive permutation group generated by transpositions is symmetric (see for example \cite[Lemma 5]{Osada1987TheGG}).
  \end{proof}

  \Cref{cor:transitive_action} in turns allows us to prove an abstract sufficient condition for the Galois group of $g$ to be symmetric:
  \begin{theorem}
    \label{thm:strategy}
    If $\Q$ is the only subextension of $L$ that is unramified away from $ab$, then the Galois group of $g$ over $\Q$ is $S_{\deg g}$.
  \end{theorem}
  \begin{proof}
    By \Cref{cor:transitive_action}, it suffices to show that $N$ acts transitively on $\Hom_{\Q}\left( L, \overline{\Q} \right)$.
    We will show it using the famous fact that the functor $\Hom\left( -, \overline{\Q} \right)$ is a contravariant equivalence between the category of \'{e}tale $\Q$-algebras and the category of finite $G_\Q$-sets.
    Since $N$ is normal in $G_\Q$, $G_\Q$ acts on $\Hom_{\Q}\left( L, \overline{\Q} \right) / N$  through $G_\Q/N$.
    In particular, the quotient map $\Hom_{\Q}\left( L, \overline{\Q} \right) \to \Hom_{\Q}\left( L, \overline{\Q} \right) / N$ is an epimorphism in the category of $G_\Q$-sets.
    Therefore, there is a corresponding embedding of fields $K \hookrightarrow L$.
    
    However, since $N$ acts trivially on $\Hom_{\Q}\left( L, \overline{\Q} \right) / N$, it follows that $K$ is fixed by $N$ and is hence ramified only above primes that divide $ab$.
    Therefore, $K = \Q$ and hence 
    \[
      \Hom_\Q\left( L, \overline{\Q} \right) / N \simeq 
      \Hom_\Q\left( \Q, \overline{\Q} \right) 
    \] 
    and is hence a singleton implying that $N$ indeed acts transitively on $\Hom_{\Q}\left( L, \overline{\Q} \right)$.
  \end{proof}

  Recall that we use $\rd\left( K \right)$ to denote the root discriminant of a number field $K$.
  We can combine this result with the bound on the discriminant to prove \Cref{thm:criterion}:
  \begin{proof}[Proof of \Cref{thm:criterion}]
    We will show that the conditions of \Cref{thm:strategy} hold.
    Let $\overline{\Q} \subsetneq K$ be a number field unramified away from $ab$.
    Assume towards a contradiction that $K \subseteq L$.
    It follows, by the contrapositive of the first assumption in the theorem statement, that there is a prime $p|ab$ such that 
    \[
      C_p \leq v_p\left( \rd\left( K \right) \right) 
    \] 
    The inequality in the statement of the theorem is equivalent to
    the statement that
    \[
      C_p > \frac{n v_p\left( ab \right) + \exrootscountbound}{\deg L}
      .
    \]
    for all $p|ab$.
    Therefore, if we denote by $d\left( L/K \right)$ the relative discriminant of $L$ over $K$, and by $N$ the norm from $K$ to $\Q$, then
    \begin{align*}
      & v_p\left( d\left( L \right) \right) = 
      v_p\left( \d\left( K \right)^{\left[ L:K \right]} N\left( \d\left( L/K \right) \right) \right) \geq \deg L \cdot v_p\left( \rd\left( K \right) \right) \geq \\
      & \deg L \cdot C_p > 
      \deg L \cdot v_p\left( ab \right) + \exrootscountbound
    \end{align*}
    which is a contradiction to \Cref{prop:discriminant_bound}.
  \end{proof}

\section{
  \label{sec:examples}
  Proof of \Cref{thm:main}
}
Recall that for each finite or infinite rational place $\p$ we view $\overline{\Q}$ as a subfield of $\overline{\Q_\p}$ and denote by $v_\p$ the standard valuation of $\overline{\Q_\p}$.
We also denote $i = \sqrt{-1}$ viewed as an element of $\C$.
To prove \Cref{thm:main} we prove first the following

\begin{proposition}
  \label{lem:fields_not_embeddable}
  Assume that $\left| ab \right| \leq 5$, $n > 8$ and that $\Q \subsetneq K$ is a number field that is unramified away from $ab$ for which 
  \[
    v_p\left( \rd K \right) \leq v_p\left( ab \right).
  \]
  Then either 
  \begin{enumerate}
    \item $\left( a,b \right) \in \left\{ \left( \pm 4, -1 \right), \left( \pm 5, -1 \right) \right\}$, or
    \item the non-cyclotomic part of $f$ is reducible over $\Q$,
    \item $K$ can not be embedded into $\Q\left[ x \right]/ \left( f \right)$.
  \end{enumerate}
\end{proposition}
\begin{remark}
  We chose the bound $\left| ab \right| \leq 5$ in \Cref{lem:fields_not_embeddable} since the LMFDB \cite{lmfdb} has a complete list of all number fields with root discriminant at most $5.989$.
  Statistical data suggests that \Cref{lem:fields_not_embeddable} is true even without excluding the cases $\left( a,b \right) = \left( \pm4, -1 \right), \left( a,b \right) = \left( \pm 5, -1 \right)$, but we failed to prove it for $K = \Q\left( \sqrt{2} \right)$ in the first case and for $K = \Q\left( \sqrt{5} \right)$ in the second one.
\end{remark}

To prove \Cref{lem:fields_not_embeddable} we present a collection of obstructions to the embeddability of $K$ into $\Q\left[ x \right]/ \left( f \right)$, some in the form of irreducibility results which will also be used in the proof of \Cref{thm:main} directly.

\begin{lemma}
  \label{lem:no_splitting_irreducibility}
  Let $K$ be a Galois number field and $p$ a rational prime and assume that:
  \begin{enumerate}
    \item $a = \pm 1$.
    \item $b$ is, up to a unit, a non trivial power of $p$.
    \item $p$ does not split in $K$.
    \item \label{item:valuation_coprimality}$
      1 = \gcd\left( 
        k, 
        v_p\left( b \right) \cdot \left[ v_p\left( K \right): \Z \right]
      \right)
    $.
    \item $n > 2$.
  \end{enumerate}
  Then one of two cases occurs
  \begin{enumerate}
    \item $f$ is irreducible over $K$.
    \item $b = \pm 2$ and $f$ factors into $x \pm 1$ times a polynomial that is irreducible over $K$.
  \end{enumerate}
\end{lemma}
\begin{proof}
  Firstly, note that by considering the Newton polygon of $f$ with respect to $p$ it follows that $f$ has $k$ roots (counted with possible multiplicity, here and in the sequel of the proof) whose $p$-adic valuation is $v_p\left( b \right) / k$ and the rest are of $p$-adic valuation $0$.
  Therefore, the $y$-coordinates of the vertices of the Newton polygon of any factor of $f$ in $K\left[ x \right]$ is in $\frac{v_p\left( b \right)}{k} \Z \cap v_p\left( K \right)$. 
  According \Cref{item:valuation_coprimality} this intersection of groups is $v_p\left( b \right)\Z$.
  Consequently, any factor of $f$ annihilates either non or all of the $k$ positive valuation roots of $f$.
  Let then $f_0 \in K\left[ X \right]$ be an irreducible monic factor of $f$ that annihilates all of the $k$ positive valuation roots of $f$.
  It suffices to show that either $f / f_0 = 1$, or $b = \pm 2$ and $f / f_0 = x \pm 1$.

  We will first show that $f_0$ is in $\Q\left[ x \right]$ and consequently so is $f / f_0$.
  Let $\sigma \in \Gal\left( \overline{\Q}/ \Q \right)$ be an automorphism of $\overline{\Q}$ and denote by $\sigma$ the induced automorphism of $\overline{\Q}\left[ x \right]$ as well.
  It suffices to show that $\sigma\left( f_0 \right) = f_0$.
  Since $p$ does not split in $K$, $\sigma\left( f_0 \right)$ is also an irreducible factor $f$ with $k$ roots of positive $p$-adic valuation. 
  However, as mentioned $f$ has only $k$ such roots and hence it has at most one such irreducible factor and hence indeed $f_0 = \sigma\left( f_0 \right)$.

  If $f / f_0 = 1$ we are done.
  Assume otherwise and let $f_1$ be some factor of $f / f_0$ that is irreducible in $\Q\left[ x \right]$ and denote by $b_0$ and $b_1$ the free coefficients of $f_0$ and $f_1$ respectively.
  Since $f_0$ annihilates all of the positive $p$-adic valuation roots of $f$ it follows that $v_p\left( b_0 \right) = v_p\left( b \right)$, and hence, since $b_0 b_1$ divides $b$ in $\Z$ and since $b$ is assumed to be a $p$-power up to a unit, it follows that $b_1$ is a rational unit, i.e. that $b_1 = \pm 1$.
  Then $f_1$ has at last one root, say $\alpha$, within the closed complex unit disk.
  By the triangle inequality it thus follows that
  \[
    \left| b \right| \leq \left| \alpha^n \right| + \left| \alpha^k \right| \leq 2
  \]
  and hence, since we assume that $b$ is up to a unit a non trivial power of $p$, it follows that $b = \pm 2$, as we set to prove, and hence that 
  \[
    \alpha^k = \alpha^n = \pm 1.
  \]
  Since $1 = \gcd\left( k,n \right)$ it follows that $\alpha = \pm 1$ or equivalently, since $f_1$ is assumed to be irreducible that $f_1 = x \pm 1$.
  Finally, since $1$ and $=-1$ can not both be roots of $f$ and since none of them can be a double root of $f$ it follows that $f / f_0 = x \pm 1$ completing our proof.
\end{proof}

\begin{lemma}
  \label{lem:imprimitivity_finite_obstruction}
  Let $K$ be a number field and let $p$ be a prime such that the inertia degree of any prime ideal of $K$ above $p$ is greater than $1$.
  Assume that there is an integer $m \in \Z$ such that 
  \[
    v_p\left( f\left( m \right) \right) = 1
  \]
  Then $K$ cannot be embedded into $\Q\left[ x \right] / \left( f \right)$.
\end{lemma}
\begin{proof}
  To show that $K$ can not be a subfield of $\Q\left[ x \right] / \left( f \right)$, it suffices to show that $f$ has a root that is contained in a purely ramified extension of $\Q_p$.
  Let then $e$ be the multiplicity of $m$ as a root of $f$ modulo $p$.
  Then there is a factorization 
    \[
      \pi\left( f \right) = \left( x-m \right)^e \overline{g} \in \mathbb{F}_p\left[ x \right]
    \]
    for some $\overline{g} \in \mathbb{F}_p\left[ x \right]$ that is coprime to $\left( x-m \right)$ where $\pi : \Z_p\left[ x \right] \to \mathbb{F}_p\left[ x \right]$ is the quotient map.
    By Hensel's lemma, there is a factorization
    \[
      f = h g \in \Z_p\left[ x \right]
    \]
    where $\pi\left( h \right) = \left( x-m \right)^e$ and $\pi\left( g \right) = \overline{g}$.
    Consider now the polynomial 
    \[
      f\left( x + m \right) = 
      h\left( x + m \right)g\left( x + m \right) = x^n + \sum_{i=0}^{n-1} c_i x^i \in \Z\left[ x \right]  
    \]
    Then 
    \[
      v_p\left( c_0 \right) = v_p\left( f\left( m \right) \right) = 1
    \]
    and 
    \[
      \pi\left( f\left( x + m \right) \right) = 
      \pi\left( h\left( x + m \right) \right) \pi\left( g\left( x + m \right) \right) =
      x^e \overline{g}\left( x + m \right) \in \mathbb{F}_p\left[ x \right]
    \]
    and hence 
    \[
      v_p\left( c_i \right) \geq 1
    \]
    for all $i < e$.
    Therefore, the Newton polygon of $f\left( x + m \right)$ with respect to $p$ has a single segment of negative slope and its slope is at most $1/e$ in absolute value.
    Therefore, $f\left( x + m \right)$ has a root, call it $\beta$, whose $p$-adic valuation is at most $1/e$.
    Therefore, since the set of roots of $f\left( x + m \right)$ that reduce to $0$ modulo $p$ is exactly the set of roots of $h\left( x + m \right)$, whose degree is $e$, it follows that the ramification degree of $\beta$ over $\Q_p$ is exactly its over all algebraic degree over $\Q_p$ and hence $\Q_p\left( \beta \right) / \Q_p$ is purely ramified.
    This shows that there is an embedding, in particular the one extending the assignment $x \mapsto \beta$, of $\Q\left[ x \right] / \left( f \right)$ into a purely ramified extension of $\Q_p$  
\end{proof}

\begin{lemma}
  \label{lem:not_embeddable_Qi_if_a2_b2}
  Assume that $\left| a \right| = \left| b \right| = 2$ and that $n > 4$.
  Then $f$ is irreducible over $\Q\left( i \right)$ where $i = \sqrt{-1}$.
\end{lemma}
\begin{remark}
  The following proof is inspired by a technique that was used first, it seems, by Ljunggren \cite{Ljungren1960} in another context and which I learned about from an expository paper by Conrad about the Irreducibility of $x^n - x - 1$ \cite{conradirreducibility}.
\end{remark}
\begin{proof}
  Denote by $\overline{\left( - \right)}$ the non trivial automorphism of $\Q\left( i \right)$.
  Assume toward a contradiction that $f$ is reducible over $\Q\left( i \right)$.
  Since it is integral and irreducible over $\Z$, it must factor as $f = f_0 \overline{f_0}$ for some irreducible 
  \[
    f_0 = x^{n/2} + \sum_{j = 0}^{n/2 - 1} c_j x^j \in \Z\left[ i \right]\left[ x \right]
    .
  \]
  
  For any polynomial $h \in K\left[ x \right]$, the Euclidean norm of $h$ is denoted by $\left\lVert h \right\rVert$ and defined to be the Euclidean norm of the coefficients' vector of $h$, which also happens to be the square root of the free coefficient of $h\left( x \right) \overline{h \left( x^{-1} \right)}$.
  In particular, since 
  \[
    f\left( x \right) \cdot \overline{f \left( x^{-1} \right)} 
    =
    f_0\left( x \right)f_0\left( x^{-1} \right) \cdot 
    \overline{\left( f_0\left( x \right)f_0\left( x^{-1} \right) \right) \circ x^{-1}}
  \]
  it follows that
  \begin{equation}
    \label{eq:euclidean_norm}
    9
    = 
    \left\lVert f \right\rVert^2
    =
    \left\lVert f_0\left( x \right)f_0\left( x^{-1} \right) \right\rVert^2.
  \end{equation}

  On the other hand, from ramification theory, it follows that $f_0$ is Eisenstein with respect to the prime ideal $\left( 1 + i \right) \subseteq \Z\left[ i \right]$, and hence $v_{2}\left( c_i \right) \geq 1/2$ for each $i = 0, \ldots, n / 2 - 1$ and $v_2\left( c_0 \right) = 1/2$.
  A direct calculation thus shows that for each $0 \leq j < n / 2$ for which $v_2\left( c_j \right) = 1/2$, the $2$-adic valuation of the $j$-th and the $-j$-th coefficients of $f_0\left( x \right)f_0\left( x^{-1} \right)$ is $1/2$ as well, and in particular the squared absolute value of each of these coefficients is at least $2$.
  Therefore, to derive a contradiction to \cref{eq:euclidean_norm} it suffices to show that there are at least three such $j$-values. 

  Since $2 = c_0 \overline{c_0}$, $0$ is one such $j$-value.
  We will deduce that there are at least two more such $j$-values by using twice the fact that for any $\alpha \in \Z\left[ i \right]$ of $2$-adic valuation $1/2$, 
  $$\alpha + \overline{\alpha} \equiv 2 \mod \left( 2 + 2i \right).$$
  Indeed, since $n$ is not coprime to $n/2$, it follows that $n/2 \neq k$ and hence the $n/2$-th coefficient of $f$ is $0$.
  However, the $n/2$-th coefficient of $f$ can be expressed as
  \[
    c_0 + \overline{c_0} + \sum_{j = 1}^{n/2 - 1} c_j \overline{c_{n/2 - j}} 
    \equiv
    c_0 + \overline{c_0} + c_{n/4} \overline{c_{n/4}}
    \equiv
    2 + c_{n/4} \overline{c_{n/4}}    
    \mod \left( 2 + 2i \right)
  \]
  when $c_{n/4}$ is to be interpreted as $0$ if $n/4$ is not an integer.
  However, for the $n/2$-th coefficient of $f$ to be $0$, it must in particular be congruent to $0$ modulo $\left( 2 + 2i \right)$ implying that $$c_{n/4} \overline{c_{n/4}} \equiv 2 \mod \left( 2 + 2i \right)$$ and hence $v_2\left( c_{n/4} \right) = 1/2$ and $n/4$ is a second such $j$-value.
  Since $n > 4$, $n$ and $n/4$ are also not coprime and hence the same argument can be repeated to show that $n/8$ is a third such $j$-value, concluding the desired contradiction.
\end{proof}

\begin{lemma}
  \label{lem:factorization_of_f}
  Assume that $n > 7, \left| a \right| = 2, \left| b \right| = 1$.
  Then one of two cases occurs:
  \begin{enumerate}
    \item $f\left( 1 \right)f\left( -1 \right) \neq 0$ and $f$ is irreducible.
    \item $f\left( 1 \right)f\left( -1 \right) = 0$ and $f$ factors into $x \pm 1$ times an irreducible factor.
  \end{enumerate}
\end{lemma}
\begin{proof}
  From \cite[p. 248]{Filaseta2019ComputationallyCP} it follows that the non cyclotomic part of $f$ is irreducible.
  However, if $\zeta$ is a root of unity and 
  $
    \left| \zeta^n \pm 1 \right| = \left| 2\zeta^k \right|
  $
  then 
  $
    \zeta^n,\zeta^k \in \left\{ \pm 1 \right\}
  $
  and hence $\zeta = \pm 1$.
  Moreover, it is easy to see that $\pm 1$ can not both be roots of $f$ and that neither of them is a double root and hence the cyclotomic part of $f$ is either $1$ or $x \pm 1$.
\end{proof}
\begin{lemma}
  \label{lem:Qi_not_embeddable_if_a2_b1}
  Assume that $n > 7, \left| a \right| = 2, \left| b \right| = 1$.
  Then $\Q\left( i \right)$ can not be embedded into $\Q\left[ x \right]/ \left( f \right)$.
\end{lemma}
\begin{proof}
  Assume first that $f\left( 1 \right)f\left( -1 \right) \neq 0$ and hence by \Cref{lem:factorization_of_f} $f$ is irreducible.
  It suffices to show that $f$ has a real root.
  However, for $f$ to not have a real root, $n$ must be even $k$ must be odd and $b$ must be $1$ and it can be verified that in such a case $0 = f\left( 1 \right)f\left( -1 \right)$ which is a contradiction.
  
  Assume now that $f\left( 1 \right)f\left( -1 \right) = 0$ and let $g$ be the non-linear factor of $f$.
  For $g$ to not have a real root, $n$ must be odd.
  However, in this case $f$ is separable over $\mathbb{F}_2$ and hence so does $g$ and hence $L:= \Q\left[ x \right] / \left( g \right)$ is unramified above $2$ and in particular can not contain $\Q\left( i \right)$.
\end{proof}

\begin{lemma}
  \label{lem:Qi_not_embeddable_if_a1_b2}
  Assume that $n > 3, \left| a \right| = 1, \left| b \right| = 2$.
  Then $\Q\left( i \right)$ can not be embedded into $\Q\left[ x \right]/ \left( f \right)$.
\end{lemma}
\begin{proof}
  If $k$ is odd the result follows from \Cref{lem:no_splitting_irreducibility} with $p = 2$.
  Assume then that $k$ is even and thus that $n$ is odd.
  If $f\left( 1 \right)f\left( -1 \right) \neq 0$ then according to \Cref{lem:no_splitting_irreducibility} $f$ is irreducible over $\Q$ and hence $\Q\left[ x \right]/ \left( f \right)$ is a number field of an odd degree and hence can not contain $\Q\left( i \right)$.
  Assume then that $f\left( 1 \right)f\left( -1 \right) = 0$.
  According to \Cref{lem:no_splitting_irreducibility} $f$ factors as $x \pm 1$ times an irreducible factor which we shall denote by $g$.

  One of two cases occurs:
  Assume first that $f$ has a root $\alpha$ of $2$-adic valuation $0$. 
  Then the $2$-adic valuation of the different ideal of $L := \Q\left[ x \right]/ \left( g \right)$ is bounded by 
  \[
    v_2\left( g'\left( \alpha \right) \right)
    \leq
    v_2\left( f'\left( \alpha \right) \right)
    =
    0.
  \]
  However, the different ideal of $\Q\left( i \right)$ is $\left( 2 \right)$ and hence $\Q\left( i \right)$ can not be embedded into $L$ and hence neither into $\Q\left[ x \right]/ \left( f \right)$.

  Assume now that all of the $n - 1$ roots of $g$ have positive $2$-adic valuation.
  Let $\alpha$ be a root of $g$.
  By considering the Newton polygon of $f$ with respect to $2$ it is clear that $v_2\left( \alpha \right) = \frac{1}{n - 1}$.
  Then again, the $2$-adic valuation of the different ideal of $L$ is bounded by 
  \[
    v_2\left( g'\left( \alpha \right) \right)
    \leq
    v_2\left( f'\left( \alpha \right) \right)
    =
    1.
  \]
  However, the $2$-adic valuation of the different ideal of $\Q\left[ i \right]$ is also $1$.
  Therefore, if we assume towards contradiction that there is an embedding $\Q\left( i \right) \hookrightarrow L$, it follows that $L$ is unramified over $\Q\left[ i \right]$ at $v_2$.
  This is a contradiction to the mentioned fact that $v_2\left( \alpha \right) = \frac{1}{n - 1}$ since by assumption $n - 1 > 2$ which is the ramification degree of $\Q\left[ i \right]$ at $v_2$.
  \end{proof}

\begin{table}
  \centering
  \begin{tabular}{ |Sc|Sc|Sc| } 
    \hline
    $\left| ab \right|$ 
    &
    \makecell{
      Minimal number fields 
      } 
     & 
     \makecell{
      Next root discriminant 
      } \\
    \hline
    $2$ & $\Q\left( \sqrt{-1} \right)$ & $2^{3/2}$ \\ 
    $3$ & $\Q\left( \sqrt{-3} \right)$ & $3^{4/3}$\\ 
    $4$ & $\Q\left( \sqrt{-1} \right), \Q\left( \sqrt{2} \right), \Q\left( \sqrt{-2} \right)$ & $> 5.989$ \\
    $5$ & $\Q\left( \sqrt{5} \right)$ & $> 5.989$ \\ 
  \hline
  \end{tabular}
  \caption{Minimal number fields satisfying the conditions of \Cref{lem:fields_not_embeddable} and the next root discriminant realized by a minimal number field unramified away from $ab$ for each value of $\left| ab \right|$ according to \cite{lmfdb}.}
  \label{table}
  \end{table}

Collecting all the above lemmas we can present the following
\begin{proof}[Proof of \Cref{lem:fields_not_embeddable}]
  The second column of \Cref{table} lists the minimal number fields satisfying the conditions of \Cref{lem:fields_not_embeddable} for each value of $\left| ab \right|$.
  It suffices to rule out the possibility of embedding each of these fields into $\Q\left[ x \right]/ \left( f \right)$, since an embedding of any other field satisfying the conditions of \Cref{lem:fields_not_embeddable} would restrict to an embedding of one of these minimal fields.
  
  The cases $\left( a,b \right) = \left( \pm 2, \pm 1 \right)$ are ruled out by \Cref{lem:Qi_not_embeddable_if_a2_b1} and the cases $\left( a,b \right) = \left( \pm 1, \pm 2 \right)$ by \Cref{lem:Qi_not_embeddable_if_a1_b2}.
  
  Assume then that $\left| ab \right| > 2$.
  We shall first treat all of the cases in which $\left| a \right| = 1$.
  In all of these cases, $f\left( 1 \right)f\left( -1 \right) \neq 0$ and hence $f$ is irreducible over $\Q$ according to \Cref{lem:no_splitting_irreducibility}.
  Therefore, and since all of the fields in \Cref{table} are quadratic, we can assume that $n$ is even and thus that $k$ is odd.
  Therefore, all of these cases can be treated by \Cref{lem:no_splitting_irreducibility} with $p$ being the unique prime factor of $b$.

  The cases $\left( a,b \right) = \left( \pm 2, \pm 2 \right)$ are resolved by \Cref{lem:not_embeddable_Qi_if_a2_b2}.
  In all of the remaining cases, namely $\left( a,b \right) = \left( \pm 3, \pm 1 \right), \left( \pm 4, 1 \right), \left( \pm 5, 1 \right)$, no root of unity can be a root of $f$ and hence, assuming that the non-cyclotomic part of $f$ is irreducible, $f$ must also be irreducible.
  Additionally, $f$ has a real root, and hence, since it is irreducible, we can cross out all of the imaginary fields from \Cref{table}.
  Therefore, the only remaining obstacles are $\left( a,b \right) = \left( \pm 4, 1 \right)$ with $K = \Q\left( \sqrt{2} \right)$ and $\left( a,b \right) = \left( \pm 5, 1 \right)$ with $K = \Q\left( \sqrt{5} \right)$.
  
  We will treat first the case $\left( a,b \right) = \left( \pm 4, 1 \right)$ and $K = \Q\left( \sqrt{2} \right)$.
  If $n$ is odd then $f$ is separable over $\mathbb{F}_2$ and hence $\mathbb{F}_2\left[ x \right]/ \left( f \right) = \mathbb{F}_2 \otimes_\Z \Z\left[ x \right] / \left( f \right)$ is a reduced ring and in particular can not contain $\mathbb{F}_2 \otimes_{\Z} \O_K$.
  However, an embedding of $K$ into $\Q\left[ x \right]/ \left( f \right)$ would induce an embedding of $\mathbb{F}_2 \otimes_{\Z} \O_K$ into $\mathbb{F}_2 \otimes_\Z \Z\left[ x \right] / \left( f \right)$ which is impossible.
  Assume then that $n$ is even and thus that $k$ is odd.
  Then by possibly replacing $f$ with $f\left( -x \right)$ we are free to assume that $a = 4$.
  Finally, this case can be dealt with by \Cref{lem:imprimitivity_finite_obstruction} with $p = 3$ and $m = 1$.

  Lastly, we treat the case $\left( a,b \right) = \left( \pm 5, 1 \right)$ and $K = \Q\left( \sqrt{5} \right)$. If $a = 5$ then we are done by \Cref{lem:imprimitivity_finite_obstruction} with $p = 7$ and $m = 1$, and if $a = -5$ then by \Cref{lem:imprimitivity_finite_obstruction} with $p = 3$ and $m = 1$ our proof is complete.
\end{proof}

\begin{proof}[Proof of \Cref{thm:main}]
  Assume first that $\left| ab \right| = 1$.
  In this case we can use \Cref{thm:strategy} directly to conclude that the Galois group of any irreducible factor of $f$ is symmetric.
  As was shown by Ljunggren \cite{Ljungren1960}, either $f$ is irreducible over $\Q$ or it factors into $x^2 \pm x + 1$ times an irreducible factor, say $g$.
  Therefore, it only remains to show that if $f$ is reducible over $\Q$ then the splitting field of $g$ over $\Q$ is linearly disjoint from the splitting field of $x^2 \pm x + 1$ over $\Q$, i.e. from $\Q\left( \sqrt{-3} \right)$.
  To do that in turn, it suffices to show that $3$ is unramified in $\Q\left[ x \right] / \left( g \right)$, for which it suffices to show that $g$ is separable over $\mathbb{F}_3$.
  However, we have already established in \Cref{cor:inertia_group_acts_as_transposition_external_settings} that $f$ has at most one double root in $\overline{\mathbb{F}_3}$, and hence, since $x^2 \pm x + 1$ has a double root there, $g$ can not have any.

  In the rest of the cases $\left| ab \right|$ is a non trivial power of a prime which we shall denote by $p$ and hence \Cref{lem:fields_not_embeddable} implies that we can apply \Cref{thm:criterion} with $C_p$ being the value recorded in the third column of \Cref{table} for the corresponding value of $\left| ab \right|$.

  If $\left| ab \right| = 2$ then according to \Cref{lem:no_splitting_irreducibility} applied to $K=\Q$ and according to \Cref{lem:factorization_of_f}, $f$ is either irreducible over $\Q$ or it factors into $x \pm 1$ times an irreducible factor.
  In both cases, it suffices to show that the Galois group of the unique non linear irreducible factor of $f$, which we shall denote by $g$, is $S_{\deg \left( g \right)}$.
  However, $\deg g \geq n - 1$ and hence a quick calculation verifies that the inequality in \Cref{thm:criterion} is satisfied for $n > 8$ and hence the Galois group of $g$ is $S_{\deg \left( g \right)}$ as required.

  We claim that in the rest of the cases $f$ is irreducible.
  Indeed, either of the following cases occurs:
  \begin{enumerate}
    \item $f$ is an Eisenstein polynomial.
    \item $\left| a \right| = 1$ and $\left| b \right| = 3,4,5$, in which case it follows from \Cref{lem:no_splitting_irreducibility} that $f$ is irreducible over $\Q$.
    \item $\left| b \right| = 1$ and $\left| a \right| = 3,5$, in which Schinzel and Mikusi{\'n}ski \cite{Schinzel1964} proved that $f$ is irreducible if $n > 2^{18}, 2^{50}$ respectively.
    \item $\left( a,b \right) = \left( \pm 4, 1 \right)$ in which case Schinzel \cite{schinzel_1966} proved that $f$ is irreducible for sufficiently large $n$.
  \end{enumerate} 
  Therefore, in all of the remaining cases, we can evaluate the inequality in \Cref{thm:criterion} with $\deg g = n$, and a quick calculation shows that it is satisfied for $n > 8$ and thus the Galois group of $f$ is $S_n$ and our proof is complete.
\end{proof}

  \bibliographystyle{alpha}
  \bibliography{refs}

\end{document}